\documentclass[11pt, oneside]{article}   	
\usepackage{geometry}                		
\usepackage{graphicx}				
\usepackage{amssymb}
\usepackage{amsthm}
\usepackage{graphicx}
\usepackage{upref}
\usepackage[active]{srcltx}
\usepackage{mathtools}
\usepackage{amsmath,amsfonts}
\usepackage{epsfig}
\usepackage[usenames, dvipsnames]{color}

\usepackage[pagebackref,colorlinks=true, linkcolor=blue, citecolor=blue]{hyperref}

\newcommand{\R}{{\mathbb R}}

\newcommand{\N}{{\mathbb N}}

\newcommand{\Np}{N^{1,p}}

\newcommand{\pip}{\varphi}
\newcommand{\eps}{\varepsilon}
\newcommand{\dmu}{d\mu}

\newcommand{\mube}{\mu_\beta}

\DeclareMathOperator{\diam}{diam}

\DeclareMathOperator{\dist}{dist}

\DeclareMathOperator{\rad}{\operatorname{rad}}

\DeclareMathOperator{\Cp}{Cap}
\DeclareMathOperator{\cp}{cap}

\def\vint{\mathop{\mathchoice%
          {\setbox0\hbox{$\displaystyle\intop$}\kern 0.22\wd0%
           \vcenter{\hrule width 0.6\wd0}\kern -0.82\wd0}%
          {\setbox0\hbox{$\textstyle\intop$}\kern 0.2\wd0%
           \vcenter{\hrule width 0.6\wd0}\kern -0.8\wd0}%
          {\setbox0\hbox{$\scriptstyle\intop$}\kern 0.2\wd0%
           \vcenter{\hrule width 0.6\wd0}\kern -0.8\wd0}%
          {\setbox0\hbox{$\scriptscriptstyle\intop$}\kern 0.2\wd0%
           \vcenter{\hrule width 0.6\wd0}\kern -0.8\wd0}}%
          \mathopen{}\int}

\theoremstyle{plain}
\newtheorem{theorem}{Theorem}[section]

\newtheorem{proposition}[theorem]{Proposition}

\theoremstyle{definition}
\newtheorem{definition}[theorem]{Definition}

\newtheorem{remark}[theorem]{Remark}

\makeatletter
\def\@biblabel#1{#1.}
\makeatother
\makeatletter
\let\Thebibliography=\thebibliography
\renewcommand{\thebibliography}[1]{\def\@mkboth##1##2{}\Thebibliography{#1}
\addcontentsline{toc}{section}{References}
\frenchspacing 
\setlength{\@topsep}{0pt}
\setlength{\itemsep}{0pt}%
\setlength{\parskip}{0pt plus 2pt}%
}
\makeatother

\newcommand{\art}[6]{{\sc #1, \rm #2, \it #3 \bf #4 \rm (#5), \mbox{#6}.}}

\newcommand{\book}[3]{{\sc #1, \it #2, \rm #3.}}
\newcommand{\artprep}[3]{{\sc #1, \rm #2, \rm #3.}}

\numberwithin{equation}{section}

\title{Potential theory and quasisymmetric maps between
compact Ahlfors regular
metric measure spaces via Besov functions: preliminary}
\author{Juha Lehrb\"ack, Nageswari Shanmugalingam\footnote{The authors thank Jeff Lindquist for
valuable discussions. 
Part of the work for this paper was conducted during the time the second author was in residence at the 
Mathematical Sciences Research Institute in Berkeley, California
as a Chern Visiting Professor,
during Spring 2022 as part of the program {\it Analysis and Geometry on Random Spaces}, which is 
funded by the NSF grant No.~1440140. She wishes to thank that august institution for its kind hospitality.
The second author was partially supported by 
the grant~DMS\#1800161 and~DMS\#2054960 of the National Science Foundation.}}

\begin{document}
\maketitle

\begin{center}
\emph{Dedicated to Professor Vladimir Maz'ya\\ for his ground-breaking contributions to potential theory.}
\end{center}

\begin{abstract}
We study Besov capacities in a compact Ahlfors regular metric measure space by means of
hyperbolic fillings of the space. This approach is applicable even if the
space does not support any Poincar\'e inequalities.
As an application of the Besov capacity estimates we show that if
a homeomorphism between two Ahlfors regular metric measure spaces 
preserves, under some additional assumptions, certain Besov classes, 
then the homeomorphism is necessarily a quasisymmetric map.
\end{abstract}

\section{Introduction}

The study of potential theory is usually directed towards Sobolev spaces of functions on Riemannian manifolds, and more
recently, Newton-Sobolev spaces of functions on complete doubling metric measure spaces supporting a Poincar\'e 
inequality. These Sobolev-type spaces of functions are associated with a gradient structure, with weak (distributional)
derivatives in the Riemannian case and minimal weak upper gradients in the metric measure space case. Such gradients
have the property that if $f$ is a function in the Sobolev-type class and $f$ is constant on an open subset of the metric
space, then the norm of the weak derivative (in the Riemannian setting) and the minimal weak upper gradient (in the
metric setting) are zero almost everywhere in that open set. In the language of Dirichlet forms and Markov processes,
this property is called \emph{strongly local} property
of the energy associated with the Sobolev classes. Tools used to study potential theory related to Sobolev spaces
include locality together with the doubling property of the measure and the Poincar\'e inequality. 
The books~\cite{AH, Maz0},
and especially~\cite[Sections~10.4.1, 13.1.2]{Maz0}, have an excellent sampling of results in potential theory in the Euclidean 
setting.

However, there are many compact
doubling metric measure spaces 
that do not have sufficient number of non-constant
rectifiable curves in order to support a Poincar\'e inequality. Examples of such spaces include the standard (thin)
Sierpi\'nski carpet and the Sierpi\'nski gasket, Rickman rug, as well as the von Koch snowflake curve~\cite[Proposition~4.5]{BoP}.
In such metric spaces
a more suitable replacement for Sobolev spaces might be Besov spaces. Unfortunately (or fortunately, depending on the
perspective) the energy associated with the Besov spaces are 
not local. In this paper we use the tools of 
hyperbolic filling and lifting of measures to that hyperbolic filling as developed in~\cite{B2S} to 
study potential theory associated 
with Besov function spaces on compact 
metric measure spaces.
We establish Besov capacitary estimates for 
various configurations of pairs of subsets of 
under the assumption that the measure is
Ahlfors $Q$-regular; see Subsection~\ref{SubSec:PI} for the definition.
A discussion regarding recent developments connecting Besov spaces of functions in Euclidean spaces
and Sobolev spaces can also be found in~\cite[Sections~10.3, 10.5]{Maz0}.

The results in this note are motivated by the study in~\cite{KKSS, KYZ}. 
The results of~\cite{KKSS} use a
characterization of Besov spaces via 
scaled Haj\l asz-type gradients from~\cite{GKZ}. Our motivation is two-fold; first, to provide an alternate proof of
the potential theoretic results in~\cite[Lemma~3.3 and Lemma~3.4]{KKSS} using the new perspective of 
hyperbolic filling that enable us to avoid the scaled Haj\l asz-type method and directly handle the 
Besov norm as in~\eqref{eq:def-Besov}, and second, to extend these capacitary estimates to spaces where the measure
is Ahlfors regular but may not support any Poincar\'e inequality. 
As an application of the capacitary estimates,
we extend at the end of this note the discussion relating Besov space preservation property and
qusiconformal maps, given in~\cite{KKSS} for Ahlfors regular spaces supporting a Poincar\'e inequality, to 
a more general class of Ahlfors regular compact metric measure spaces that may not support a Poincar\'e inequality
but are linearly locally path connected.
However, since we do not assume that the metric spaces support a Poincar\'e inequality, we assume a stronger condition 
on the homeomorphism, namely that it is quasisymmetric. We show that homeomorphisms between two
Ahlfors regular (but not necessarily of the same dimension) compact metric measure spaces that are linearly locally
connected, have the property that if they preserve certain Besov classes under composition, with control over the Besov
norms, then the mapping is necessarily quasisymmetric. This is the content of Theorem~\ref{thm:QS-Besov}.

We do not know whether every quasisymmetric map between two compact Ahlfors regular linearly
locally path connected metric measure spaces with the same regularity dimension preserves 
certain Besov spaces, as lacking knowledge of a suitable Poincar\'e inequality, we do not know that such 
maps preserve measure densities. We point out that a lack of absolute continuity of the pull-back measure 
does not on its own indicate that the quasisymmetric map is not a Besov space morphism, as from~\cite[Proposition~13.3]{B2S} we
know that every Besov function can be modified on a null set to be made quasicontinuous.
This is supported by some preliminary results, because for certain compact spaces such as Cantor sets, some partial
results are known, see~\cite[Section~8]{BBGS}.

\section{Preliminaries}

This section is devoted to describing the background notions used in this note, with the setting 
considered here delineated in Subsection~\ref{subSec:setting}.

\subsection{Newton-Sobolev spaces}\label{SubSec:Newt}

Let $1\le p<\infty$.
When $\Omega$ is an $n$-dimensional 
Euclidean (or a Riemannian) domain and $f\in L^p(\Omega)$, we say that $f$ is in the Sobolev
class $W^{1,p}(\Omega)$ if $f$ has a weak derivative $\nabla f\in L^p(\Omega:\R^n)$. Note that 
if $f$ is of class $C^1(\Omega)$, then for each compact rectifiable curve $\gamma$ in $\Omega$ we have
\begin{equation}\label{eq:up-gr}
|f(y)-f(x)|\le \int_\gamma g\, ds,
\end{equation}
with $g=|\nabla f|$,
where $x$ and $y$ denote the two end points of $\gamma$. However, if $f$ is not of class $C^1(\Omega)$,
a weaker analog of this holds, see~\cite{Vaisala}, namely, there is a family $\Gamma_f$ of compact rectifiable curves
in $\Omega$ such that whenever $\gamma$ is a compact rectifiable curve in $\Omega$ that \emph{does not} belong
to $\Gamma_f$, then~\eqref{eq:up-gr} holds. Moreover, the family $\Gamma_f$ is of $p$-modulus zero, that is, there
is a non-negative Borel measurable function $\rho\in L^p(\Omega)$ such that $\int_\gamma\rho\, ds=\infty$ for each
$\gamma\in\Gamma_f$.

This is the starting point for the theory of Newton-Sobolev functions on
metric measure spaces where weak derivatives do not make sense. 
Let $Y$ be a metric space equipped with a Radon measure $\mu$, and let $f$ be a function on $Y$. We say that
a non-negative Borel measurable function $g$ is a \emph{$p$-weak upper gradient} of $f$ if there is a family $\Gamma$ of
non-constant compact rectifiable curves in $X$ (possibly empty) such that there
is a non-negative Borel measurable function $\rho\in L^p(Y)$ satisfying $\int_\gamma\rho\, ds=\infty$ for each
$\gamma\in\Gamma_f$, and for each non-constant compact rectifiable curve $\gamma$ in $Y$ 
with $\gamma\not\in\Gamma$, the pair $f$ and $g$ 
satisfies~\eqref{eq:up-gr}. We set $\widetilde{N}^{1,p}(Y)$ to be the collection of all functions $f$ such that
$\int_Y|f|^p\, d\mu<\infty$ and $f$ has a $p$-weak upper gradient $g\in L^p(Y)$. Note that we do not ask that
$f\in L^p(Y)$ as elements of $L^p(Y)$ are equivalence classes of functions that agree outside measure-null sets,
but the existence of a weak upper gradient from $L^p(Y)$
may fail if we modify $f$ on a set of measure zero.
The Newton-Sobolev space $N^{1,p}(Y)$ is set to be the collection $\widetilde{N}^{1,p}(Y)/\sim$ of equivalence classes,
where two functions $f_1,f_2\in \widetilde{N}^{1,p}(Y)$ are equivalent, $f_1\sim f_2$, if $\Vert f_1-f_2\Vert_{N^{1,p}(Y)}=0$. Here
\[
\Vert f\Vert_{N^{1,p}(Y)}^p:=\int_X|f|^p\, d\mu+\inf_g\int_Xg^p\, d\mu,
\]
with the infimum taken over all $p$-weak upper gradients $g$ of $f$. For $1\le p<\infty$,
for each $f\in N^{1,p}(Y)$ there is a minimal $p$-weak upper
gradient $g_f\in L^p(Y)$, that has the smallest $L^p$-norm of all $p$-weak upper gradients of $f$. We refer the interested
reader to~\cite{HKST} and~\cite{BBbook} for more on Newton-Sobolev spaces.

\subsection{Poincar\'e inequalities and doubling measures}\label{SubSec:PI}

For $1\le p<\infty$, we say that 
the metric measure space
$(Y,d,\mu)$ supports a $p$-Poincar\'e inequality if there is a constant $C>0$ such that 
\[
\vint_B|f-f_B|\, d\mu\le C\, \rad(B)\, \left(\vint_Bg^p\, d\mu\right)^{1/p}
\]
whenever $g$ is a $p$-weak upper gradient of $f$ in $Y$ and $B$ is a ball in $Y$. 
Here we use the notation 
\[
u_B=\vint_B u\, d\mu=\mu(B)^{-1} \int_B u\, d\mu
\]
for the mean-value integral over $B$.
The validity of a $p$-Poincar\'e inequality
immediately implies that $Y$ is connected. If $\mu$ is in addition doubling, and $Y$ is locally compact, then $Y$ is quasiconvex,
that is, for each $x,z\in Y$ there is a curve $\gamma$ in $Y$ with end points $x, y$ and length $\ell(\gamma)\le C\, d(x,z)$ with
$C$ independent of $x,z$. This quasiconvexity property was first proved in~\cite{Che} in the context of complete
metric measure spaces, but see~\cite[Theorem~3.1]{DJS} for the corresponding proof for locally compact metric measure spaces.

Recall that a Radon measure
$\mu$ is doubling if there is a constant $C_d\ge 1$ such that whenever $y\in Y$ and $r>0$, we have
\[
0< \mu(B(x,2r))\le C_d\, \mu(B(x,r)) <\infty.
\]
We say that $\mu$ is Ahlfors $Q$-regular for some $Q>0$ if there is a constant $C>0$ such that whenever
$x\in Y$ and $0<r<2\diam(Y)$,
\[
\frac{r^Q}{C}\le \mu(B(x,r))\le C\, r^Q.
\]
Note that Ahlfors $Q$-regular measures are comparable to the $Q$-dimensional Hausdorff measure $\mathcal{H}^Q$
(see Subsection~\ref{SubSec:Cap}).

See~\cite{HKST} and~\cite{BBbook} for more details on analysis on doubling metric measure spaces 
supporting Poincar\'e inequalities.

\subsection{Besov spaces}\label{SubSec:Besov}

The primary focus of our note
is on the Besov spaces. These were
initially formulated by O.~V.~Besov in order to study Sobolev extension and restriction theorems
for somewhat smooth Euclidean domains, see for instance~\cite{Besov, Maz0, MazPobTr}.
Let $1\le p<\infty$ and $0<\theta<1$. A function $u\in L^p(Y)$ is said to be in the Besov space $B^\theta_{p,p}(Y)$ if
its \emph{Besov energy}
\begin{equation}\label{eq:def-Besov}
\Vert u\Vert_{B^\theta_{p,p}(Y)}^p:=\int_X\int_X\frac{|u(x)-u(z)|^p}{d(x,z)^{\theta p}\mu(B(x,d(x,z)))}\, d\mu(z)\, d\mu(x)
\end{equation}
is finite. Unlike the Newton-Sobolev functions, an arbitrary perturbation of a Besov function on a set of measure zero gives an equivalent Besov function.

If $(Y,\mu)$ is a doubling metric measure space supporting a $p$-Poincar\'e inequality, then $B^\theta_{p,p}(Y)$ is
obtained via a real interpolation of $N^{1,p}(X)$ with $L^p(X)$, see~\cite{BeSharp} for the Euclidean setting 
and~\cite{GKS} for more on this in the metric setting. 
However, in this note we are
not interested in the interpolation properties connecting Newton-Sobolev spaces to Besov spaces, but in the
trace properties. Jonsson and Wallin studied the trace relationship between Sobolev classes on Euclidean spaces
and  Ahlfors regular compact subsets of the Euclidean spaces, see~\cite{JW80, JW84}.

If $Y$ is a non-complete, locally compact metric measure space, we 
set $\partial Y:=\overline{Y}\setminus Y$. Here $\overline{Y}$ is the metric completion of $Y$, obtained
by considering equivalence classes of Cauchy sequences in $Y$; hence $\partial Y$ consists of equivalence classes
of Cauchy sequences in $Y$ that do not converge in $Y$.
Observe that if $Y$ is locally compact, then necessarily $Y$ is an open subset of $\overline{Y}$. 
Suppose that $\nu$ is a Borel measure on $\partial Y$ and that $\partial Y$ is proper. We say that 
$B^\theta_{p,p}(\partial Y)$ is the trace space of $N^{1,p}(Y)$ if there is a bounded operator
\[
T:N^{1,p}(Y)\to B^\theta_{p,p}(\partial Y)
\]
and a bounded linear extension operator 
\[
E:B^\theta_{p,p}(\partial Y)\to N^{1,p}(Y)
\]
such that 
\begin{enumerate}
\item $Tu=\overline{u}\vert_{\partial Y}$ whenever $u$ is a Lipschitz function on $Y$; here $\overline{u}$ is the unique continuous
extension of $u$ to $\partial Y$,
\item $T\circ E$ is the identity map on $B^\theta_{p,p}(\partial Y)$.
\end{enumerate} 

The subject of traces of Sobolev functions in Euclidean domains dates back to the work of Besov, Gagliardo, Jonsson, and Wallin
\cite{Besov, Besov2, Gag, MazS,  JW80, JW84}. 
The canonical textbook of Maz'ya~\cite{Maz0} contains a nice discussion on traces in Chapter~11, while~\cite{MazPob, MazPobTr}
contain results linking traces of Sobolev spaces to Besov-type spaces in certain Euclidean domains. See also
the text~\cite{MP} for a general treatment of boundary values of Sobolev functions on ``bad" Euclidean domains;
these are merely a few papers on the topic from a vast literature on traces, as we cannot hope to list all papers on the 
topic of traces here.
We refer interested readers to~\cite{GKS, MalyBesov, MSS, B1S} for more on 
Besov spaces as traces of Newton-Sobolev spaces in the metric setting, and to~\cite{Tai} for connections to other
expressions of Besov spaces.

\subsection{Hyperbolic fillings and uniformization}\label{SubSec:HypFill}

Throughout this note, $(Z,d_Z,\nu)$ is a compact doubling  
metric measure space, and without loss of generality we may assume that $0<\diam(Z)<1$.
For $\alpha>1$ and $\tau>1$, we construct a Gromov hyperbolic space $X$ from 
$Z$ as a graph. For each non-negative integer $n$ we set $A_n$ to be a maximal $\alpha^{-n}$ -separated subset of $Z$ such that
$A_n\subset A_{n+1}$ for each $n\in\N_0$. The vertex set of the graph $X$ is the set
$\bigcup_{n\in\N} \{n\}\times A_n$. Two vertices $v=(n,x_n)$ and $w=(m,x_m)$ are neighbors if $v\ne w$, $|n-m|\le 1$,
and $B(x_n,\tau \alpha^{-n})\cap B(x_m,\tau \alpha^{-m})$ is non-empty if $n\ne m$ and 
$B(x_n, \alpha^{-n})\cap B(x_m, \alpha^{-m})$ is non-empty if $n=m$. We consider each pair of neighbors to be connected
with an edge that is an interval of unit length. There is only one vertex $p_0$ corresponding to the level $n=0$, that is,
$\{0\}\times A_0=\{p_0\}$.

Variants of hyperbolic fillings have been constructed in~\cite{BuSch, BP, BSch, BS, B2S, BSS}, but the one described above is from~\cite{B2S}
where it was also shown that $X$ is a Gromov hyperbolic space and that with $\eps=\log(\alpha)$, the uniformization $X_\eps$
of $X$ as in~\cite{BHK} yields a uniform space such that $Z$ is biLipschtiz equivalent to $\partial X_\eps$.  Here, the uniformization
is accomplished via the modified metric $d_\eps$ given by
\[
d_\eps(x,y)=\inf_\gamma \int_\gamma e^{-\eps d(\gamma(t),p_0)}\, ds(t)
\]
with the infimum over all rectifiable curves $\gamma$ in $X$ with end points $x$ and $y$. Recall that $X_\eps$ is a uniform space
if there is a constant $A\ge 1$ such that
for each pair of points $x,y\in X_\eps$ there is a curve $\gamma$ in $X_\eps$ with end points $x$ and $y$ such that 
$\ell_\eps(\gamma)\le A\, d_\eps(x,y)$ and for each $z\in \gamma$, 
\[
\min\{\ell(\gamma_{x,z}), \ell(\gamma_{z,y})\}\le A\, \delta_\eps(z).
\]
Here, $\gamma_{x,z}$ and $\gamma_{z,y}$ denote subcurves of $\gamma$ with end points $x,z$ and $z,y$ respectively, and
\[
\delta_\eps(z):=\dist_{d_\eps}(z,\partial X_\eps):=\inf_{w \in \partial X_{\eps}}d_{\eps}(z, w).
\]
When $Z$ is equipped with a measure $\nu$, we can lift up this measure to a measure
$\mu_+$ on $X$ by setting balls of radius $1$ centered at
vertices $v=(n,x)$ to have measure equal 
to $\nu(B(x,\alpha^{-n}))$. For each $\beta>0$ we can
uniformize this measure to obtain a measure $\mube$ on $X_\eps$ by setting
$d\mube(v)=e^{-\beta d(v,p_0)}\, d\mu_+(v)$. 
This gives us a one-parameter family of lifted measures on $X_\eps$,
first constructed in~\cite{B2S}. 

Recently Clark Butler extended the construction of hyperbolic fillings from compact doubling metric spaces to complete doubling
metric spaces that are unbounded, see~\cite{Bu1, Bu2, Bu3}. It was shown in~\cite{Bu1} that trace and extension theorems
similar to the ones in~\cite{B2S} hold even for the unbounded setting. In this note we focus on compact spaces $Z$,
but point out that with minimal effort the results here can be extended to unbounded complete doubling metric measure 
spaces as well by using the tools of~\cite{Bu1, Bu2}.

\subsection{Uniformized measure $\mu_\beta$ and connection to $\nu$}\label{subSec:setting}

For $\beta>0$ let $\mube$ be the uniformized lift of $\nu$ to $X_\eps$ as constructed
in~\cite{B2S} and described in Subsection~\ref{SubSec:HypFill}. 
From the results in~\cite{B2S} we know that the metric measure space $(X_\eps,d_\eps,\mube)$ is
a doubling metric measure space supporting the best possible Poincar\'e inequality, namely the $1$-Poincar\'e inequality.

There is a relationship between $\nu$ and $\mube$; whenever $z\in Z$ and $0<r\le \diam(Z)$,
we have, by~\cite[Theorem~10.3]{B2S} and the doubling property of $\mube$, that
\begin{equation}\label{eq:co-dim}
\mube(B(z,r))\simeq r^{\beta/\eps} \nu(B(z,r)).
\end{equation}
We treat $\nu$ as a measure on $\overline{X}$, obtained by extending $\nu$ from $Z=\partial X_\eps$ to
$X$ by zero.

\begin{proposition}[{\cite[Theorem~1.1 and Theorem~10.2]{B2S}}]\label{prop:TraceExt}
With the choice of $\alpha$ and $\eps$ as above, the uniformized space $X_\eps$,
equipped with the metric $d_\eps$ and the measure $\mube$, is doubling and supports a $1$-Poincar\'e inequality.
Moreover, for the choice $\theta=1-\beta/(\eps p)$, the Besov space $B^\theta_{p,p}(Z)$ is the trace space of
$N^{1,p}(X_\eps)$.
\end{proposition}

The above proposition is a key tool for us in this note. We will exploit this 
identification of $B^\theta_{p,p}(Z)$ with the trace of $N^{1,p}(X_\eps)$ frequently.
The fine properties of 
functions in $N^{1,p}(X_\eps,\mube)$ follow from the results of~\cite{HKST, BBbook} thanks to the doubling property
of $\mube$ and the support of the $1$-Poincar\'e inequality. While $N^{1,p}(X_\eps)$ also depends on the 
choice of $\beta$ in defining the measure on $X_\eps$, we will suppress this dependance in our notation as 
we fix $\theta$ and $p$, and hence $\beta$ in this note.

Also
the following construction of the extension $Eu$ of $u\in B^\theta_{p,p}(Z)$ to $X_\eps$ 
will be important. 
In~\cite[Theorem~12.1]{B2S}, the extension 
is constructed by first defining $Eu((n,z))$, $z\in A_n$,
by 
\begin{equation}\label{eq:Def-Eu}
Eu((n,z))=\vint_{B(z,\alpha^{-n})} u\, d\nu,
\end{equation}
and then extending $Eu$ linearly (with respect to the uniformized metric $d_\eps$) to the edges that make up the 
graph $X_\eps$. It is shown in~\cite[Theorem~12.1]{B2S}
that $Eu\in N^{1,p}(X_\eps,\mu_\beta)$ when $\theta=1-\beta/(p\eps)$,
with $TEu=u$ $\nu$-a.e.~in $Z$, and moreover
\begin{equation*}
\int_{\overline{{X}}_{\epsilon}} |Eu|^p \dmu_{\beta} \lesssim \int_{Z} |u|^p\, d\nu
\end{equation*}
and
\begin{equation}\label{eq:norm-gEu}
\int_{\overline{{X}}_{\epsilon}}g_{Eu}^p \dmu_{\beta} \lesssim \Vert u\Vert^{p}_{B^{\theta}_{p, p}(Z)}.
\end{equation}

\subsection{Capacities and Hausdorff content}\label{SubSec:Cap}

As mentioned in Subsection~\ref{SubSec:Newt}, functions in $N^{1,p}(Y)$ cannot be arbitrarily modified on general sets of
measure zero. Therefore, to study fine properties of such functions, we need a finer notion than null measure, and this 
is one purpose of the notion of capacity.

Let $(Y,d,\mu)$ be a metric measure space with $\mu$ a Radon measure. Given a set $E\subset Y$ and $1\le p<\infty$,
we set the Newton-Sobolev $p$-capacity of $E$ to be the number
\[
\Cp_{N^{1,p}(Y)}(E):=\inf_u\Vert u\Vert_{N^{1,p}(Y)}^p,
\]
where the infimum is over all functions $u\in \widetilde{N}^{1,p}(Y)$ satisfying $u\ge 1$ on $E$. It follows from the 
results of~\cite{Sh1, HKST} that Newton-Sobolev functions can be arbitrarily perturbed
only on sets of capacity zero.

On the other hand, Besov functions can
be perturbed arbitrarily on sets of measure zero. For this reason
the Besov capacity of $E$ is set to be
\[
\Cp_{B^\theta_{p,p}(Y)}(E):=\inf_u \int_Y|u|^p\, d\mu+\Vert u\Vert_{B^\theta_{p,p}(Y)}^p
\]
with infimum over all $u\in B^\theta_{p,p}(Y)$ such that $u\ge 1$ on a \emph{neighborhood of} $E$.

Related to the above two capacities there is a notion of \emph{relative capacity} of a condenser $(E,F; Y)$. If $E,F\subset Y$, then
\[
\cp_{N^{1,p}(Y)}(E, F):=\inf_u \int_Yg_u^p\, d\mu
\]
where the infimum is over all $u\in N^{1,p}(Y)$ satisfying $u\ge 1$ in $E$ and
$u\le 0$ in $F$, 
and $g_u$ the minimal $p$-weak upper gradient of $u$ as described 
at the end of Subsection~\ref{SubSec:Newt}. 
Similarly,
\[
\cp_{B^\theta_{p,p}(Y)}(E, F):=\inf_u \Vert u\Vert_{B^\theta_{p,p}(Y)}^p,
\]
where the infimum is over all $u\in B^\theta_{p,p}(Y)$ satisfying $u\ge 1$ in a neighborhood of
$E$ and $u\le 0$ in a neighborhood of $F$. 
Note that if $F\subset F_1$ and $E\subset E_1$, then
\[
\cp_{B^\theta_{p,p}(Y)}(E, F)\le \cp_{B^\theta_{p,p}(Y)}(E_1, F_1).
\]

Returning to our setting, it was shown in~\cite{B2S} 
that $N^{1,p}(X_\eps)=N^{1,p}(\overline{X}_\eps)$ and 
that when $E\subset Z$,
\[
\Cp_{N^{1,p}(\overline{X}_\eps)}(E)\simeq \Cp_{B^\theta_{p,p}(Z)}(E).
\]
It was shown there moreover that if $\Cp_{N^{1,p}(\overline{X}_\eps)}(E)=0$ then necessarily $\nu(E)=0$.
Note here that the statement holds regardless of the value of $\beta>0$ that generated the measure $\mube$ on $X_\eps$,
provided that $\beta$ is chosen so that $\theta=1-\tfrac{\beta}{\eps p}$.

Sobolev capacity is associated with Hausdorff content, as seen for example in~\cite[Section~1.1.18]{Maz0} in the Euclidean setting
and~\cite[Theorem~2.26]{HKM} in
Euclidean domains equipped with admissible weights. Given a set
$E\subset Y$,  $0<\alpha<\infty$, and $0<\tau\le \infty$, the $\alpha$-dimensional Hausdorff content of $E$ at scale $\tau$ is the
number
\[
\mathcal{H}^\alpha_\tau(E):=\inf_{(B_i)_{i\in I\subset \N}}\sum_{i\in I} \diam(B_i)^\alpha,
\]
where the infimum is over all countable covers $(B_i)_{i\in I\subset\N}$
of the set $E$, by balls $B_i$, such that for 
each $i\in I$ we have $\diam(B_i)<\tau$. 
The $\alpha$-dimensional Hausdorff measure of $E$ is then given by
\[
\mathcal{H}^\alpha(E):=\lim_{\tau\to 0^+}\mathcal{H}^\alpha_\tau(E).
\]

Hausdorff measures are a natural 
metric tool to use in an Ahlfors $Q$-regular space $Y$ to analyze Sobolev capacities.
For instance, if we assume 
in addition that $Y$ is complete, unbounded and supports a $p$-Poincar\'e inequality, with $1<p\le Q$, then it
follows from the results in \cite{Costea} that
if $\Cp_{N^{1,p}(Y)}(E)=0$, then $\mathcal{H}^s_\infty(E)=0$ for every $s>p-Q$, and conversely, 
if $\mathcal{H}^{Q-p}_\infty(E)=0$ (or even $\mathcal{H}^{Q-p}_\infty(E)<\infty$, when $1<p<Q$),
then $\Cp_{N^{1,p}(Y)}(E)=0$. 
We refer the interested reader to~\cite[Section~4.7.2]{EG} for the Euclidean
setting. Additional information can be found in~\cite[pages~28, 760]{Maz0}. In more general doubling metric measure
spaces co-dimensional Haudorff measures are more useful in controlling Sobolev capacites, see for 
instance~\cite[Proposition~3.11, Section~8]{GKS2022}, 
and relative capacities, see e.g.~\cite[Propositions~4.1 and~4.3]{Le}.

\section{Besov capacitary estimates}

In studying quasisymmetric mappings between metric spaces, there are two types of configurations that play a
key role. The first type of configuration is that of an annulus $B(x,R)\setminus B(x,r)$, and the associated condenser
is the triplet $(\overline{B}(x,r), X\setminus B(x,R), X)$ for $0<r<R\le \text{diam}(X)/2$. The second type of 
configuration arises from considering two compact continua $E,F$ contained in a ball $B(x,R)$ with
$\min\{\text{diam}(E),\text{diam}(F)\}\ge R/C$, and the associated condenser is $(E,F, X)$. We consider these two
configurations in the two subsections of this section.

We assume throughout this section  that the measure $\nu$ on
$Z$ is Ahlfors $Q$-regular for some $Q>0$. The results of this section are
modeled after \cite[Lemma~2.4 and Lemma~2.3]{KKSS} and~\cite{HK}.

\subsection{Relative Besov capacitary estimates for annular rings}\label{SubSec:2-Rings}

In this subsection we consider annular rings in $Z$, namely sets of the form $E=\overline{B}(x_0,r)$ and
$F=Z\setminus B(x_0,R)$ for $x_0\in Z$ and $0<r<R$. 
An analog of Case~2 of the following theorem for 
relative Newton-Sobolev capacity $\cp_{N^{1,Q}(Z)}$ can be found in~\cite[Lemma~3.14]{HK}.

\begin{theorem}\label{thm:Ahlfors-Annulus}
Assume that $Z$ is a compact metric space
and that $\nu$ is an Ahlfors $Q$-regular measure on $Z$, for some $Q>0$.
Let $1< p<\infty$ and $0<\theta<1$, and
suppose that $0<r<R/2$ and $x_0\in Z$. Then 
\[
\cp_{B^\theta_{p,p}(Z)}(\overline{B}(x_0,r), Z\setminus B(x_0,R))\le \xi(R)\, \Xi(r)\, \Psi(R/r),
\]
where 
\begin{enumerate}
\item if $p\theta>Q$, then $\xi(R)\simeq R^{Q-\theta p}$, $\Xi(r)=1$ and $\Psi(R/r)=1$.
\item if $p\theta=Q$, then $\xi(R)\simeq\Xi(r)\simeq 1$ and $\Psi(R/r)=\log(R/r)^{1-p}$.
\item if $p\theta <Q$, then $\xi(R)=1$, $\Xi(r)\simeq r^{Q-\theta p}$, and $\Psi(R/r)=1$.
\end{enumerate}
Therefore, when $p\theta=Q$
or when $p\theta<Q$ and $R\le 1$
the Besov capacity of the condenser $(\overline{B}(x_0,r), Z\setminus B(x_0,R))$ is at most $\Psi(R/r)\simeq \log(R/r)^{-\tau}$
for some $\tau\in\{p, p-1\}$.
\end{theorem}

\begin{proof}
We will utilize the hyperbolic filling here to give an alternate proof than the one in~\cite{KKSS}. We fix $\theta$ with
$0<\theta<1$ and choose $\beta>0$ such that $\theta=1-\beta/(\eps p)$, and consider the space $(X_\eps, d_\eps,\mube)$ 
as described in Subsection~\ref{SubSec:HypFill}. Then $B^\theta_{p,p}(Z)$ is the trace space of $N^{1,p}(X_\eps,\mube)$,
as explained in Proposition~\ref{prop:TraceExt}.

In Case~1, that is, when $p\theta>Q$, we consider the test function $u$ given by
\[
u(x)=\left(1-\frac{2\text{dist}(x,B(x_0,R/2))}{R}\right)_+.
\]
Then $u=1$ on $B(x_0,r)\subset B(x_0,R/2)$, $u=0$ on $X_\eps\setminus B(x_0,R)$,
and $u$ is $2/R$-Lipschitz continuous. Therefore
\[
\int_{X_\eps}g_u^p\, d\mube\lesssim \mube(B(x_0,R)\setminus B(x_0,R/2))\, \left(\frac{2}{R}\right)^p.
\]
Note that the balls considered in the above estimate are all centered at points in $Z=\partial X_\eps$, and so we are
in the realm of~\eqref{eq:co-dim}.
Using the facts that $\mube(B(x_0,R)\setminus B(x_0,R/2))\lesssim R^{Q+\beta/\eps}$ and $\theta=1-\beta/(p\eps)$,
we obtain
\[
\int_{X_\eps}g_u^p\, d\mube\lesssim R^{Q-p\theta}.
\]

In Case~3, that is, when $p\theta<Q$, we instead consider the function $u$ given by
\[
u(x)=\left(1-\frac{\text{dist}(x,B(x_0,r))}{r}\right)_+
\]
and note that $u=1$ on $B(x_0,r)$ and $u=0$ on $X_\eps\setminus B(x_0,2r)$. Thus we see that
\[
\int_{X_\eps}g_u^p\, d\mube\lesssim r^{Q-\theta p}.
\]

In both of these cases, by~\cite[Theorem~11.1(11.2)]{B2S}, with $u$ also denoting the trace of $u$ to $Z$ 
(and as $u$ is Lipschitz continuous,
this is a pointwise identification), we have the desired upper bound for the Besov capacity as well.

Finally, in Case~2 ($p\theta=Q$) we define the function $u$ on $\overline{X}_\eps$ by
\[
u(x):=\min\bigg\lbrace \left(\frac{\log(R/d(x,x_0))}{\log(R/r)}\right)_{+}, 1\bigg\rbrace.
\]
Then by the chain rule for upper gradients (see~\cite[(6.3.19)]{HKST} or
\cite[Theorem~2.16]{BBbook}) and by the fact that $1$ is an upper gradient of the distance function,
we see that
\begin{equation}\label{eq:g-sub-u}
g_u(x)\le \frac{1}{\log(R/r)}\, \frac{1}{d(x,x_0)}\, \chi_{B(x_0,R)\setminus B(x_0,r)}(x).
\end{equation}
Again 
by~\cite[Theorem~11.1(11.2)]{B2S},  
we have
\[
\Vert u\Vert_{B^\theta_{p,p}(Z)}^p\lesssim \int_{X_\eps}g_u^p\, d\mube.
\]
Hence it suffices to obtain integral estimates for $g_u$. Note that
$B(x_0,R)\setminus B(x_0,r)\subset\bigcup_{j=0}^{n_R}B(x_0,2^{j+1}r)\setminus B(x_0,2^jr)$ where 
$n_R$ is the smallest positive integer such that $2^{n_R}r\ge R$. We have $n_R\simeq \log(R/r)$. Then by the
bound on $g_u$ in~\eqref{eq:g-sub-u} and by~\eqref{eq:co-dim},
\begin{align*}
\int_{X_\eps}g_u^p\, d\mube  &\le \log(R/r)^{-p} \sum_{j=0}^{n_R}\int_{B(x_0,2^{j+1}r)\setminus B(x_0,2^jr)}
\frac{1}{d(x,x_0)^p}\, d\mube(x)\\
 &\lesssim \log(R/r)^{-p} \sum_{j=0}^{n_R}\frac{\mube(B(x_0, 2^jr))}{(2^jr)^p}\\
 &\simeq \log(R/r)^{-p} \sum_{j=0}^{n_R} \frac{\nu(B(x_0,2^jr))}{(2^jr)^{p-\beta/\eps}}\\
 &\simeq \log(R/r)^{-p} \sum_{j=0}^{n_R} \frac{1}{(2^jr)^{p-Q-\beta/\eps}}
 \ \ = \ \ \log(R/r)^{-p}\, n_R,  
\end{align*}
where the last equality followed from the identity $\theta=1-\beta/(\eps p)$ together with $p\theta=Q$,
and the penultimate estimate came from the 
assumption that $\nu$ is Ahlfors $Q$-regular.  
Since  
$n_R\simeq \log(R/r)$, we have
\[
\int_{X_\eps} g_u^p\, d\mube\lesssim \log(R/r)^{-p} n_R \simeq \log(R/r)^{1-p},
\]
verifying the claim in Case~2.
\end{proof}

\begin{remark}
Note that in Case~2, the capacity of the annulus tends to zero as $R/r\to\infty$. In Case~3
then the capacity of the annulus tends to zero as $r\to 0$. This perspective plays a key role in the study of homeomorphisms
that induce Besov space morphisms, and their relationship to local quasisymmetry and metric quasiconformality, see Section~\ref{Sec:QC} below.
\end{remark}

We record also the following converse of Theorem~\ref{thm:Ahlfors-Annulus}.
These bounds are 
not needed in the later results, but they in particular show that the 
estimates in Theorem~\ref{thm:Ahlfors-Annulus} are often optimal. We refer the interested reader to~\cite[Lemma~9.3.6]{HKST}
and~\cite[Sections~6 and~7]{BBL}
for the analogous estimates for Sobolev capacity in doubling metric measure spaces supporting a $p$-Poincar\'e inequality.

\begin{theorem}\label{thm:Ahlfors-Annulus-rev}
Assume that $Z$ is a compact metric space
and that $\nu$ is an Ahlfors $Q$-regular measure on $Z$, for some $Q>0$.
Let $1< p<\infty$ and $0<\theta<1$, and
suppose that $x_0\in Z$ and
$0<r<R<\diam(Z)/4C_0$ for suitably large constant $C_0>2$. Then 
\[
\cp_{B^\theta_{p,p}(Z)}(\overline{B}(x_0,r), Z\setminus B(x_0,R))\ge \xi(R)\, \Xi(r)\, \Psi(R/r),
\]
where $\xi, \Xi$, and $\Psi$ are as in Theorem~\ref{thm:Ahlfors-Annulus},
and in the case $p\theta = Q$
we assume in addition that $r\le R/2$.
\end{theorem}

\begin{proof}
Fix $0<\theta<1$ and
choose $\beta>0$ so that $\theta =1-\beta/(p \eps)$.
Let $u\in B^\theta_{p,p}(Z)$ be such that $u=1$ 
in a neighborhood of $\overline{B}(x_0,r)$ and $u=0$ in a neighborhood of $Z\setminus B(x_0,R)$.

Let $Eu$ be the extension of $u$ to the 
uniformization $X_\eps$ of the
hyperbolic filling $X$ of $Z$ as 
explained in Section~\ref{subSec:setting}. 
Then
$Eu \in \Np(\overline{X}_{\epsilon}, \mu_{\beta})$, and by~\eqref{eq:norm-gEu} we have
\[
\int_{\overline{{X}}_{\epsilon}}g_{Eu}^p \dmu_{\beta} \lesssim \Vert u\Vert^{p}_{B^{\theta}_{p, p}(Z)}.
\] 
From the way the extension $Eu$ is defined in~\cite[Theorem~12.1]{B2S}, see~\eqref{eq:Def-Eu},
it follows that $Eu=1$ on $\overline{B}_\eps(x_0,r/(\tau\alpha))$ and
$Eu=0$ on $\overline{X_\eps}\setminus B_\eps(x_0,\tau\alpha R)$;
here the parameters $\tau>1$ and $\alpha>1$ are as in Section~\ref{SubSec:HypFill},
and the subscript $\eps$ in $B_\eps$ refers to the fact that these balls are with respect
to $\overline{X_\eps}$. We ensure that $C_0\ge 2\tau\alpha$.

Write $r'=r/(\tau\alpha)$ and $R'=\tau\alpha R$. By our assumption on $C_0$, we know that $R'<\diam(Z)/4$.
Then $Eu$ is a test function for the capacity 
$\cp_{N^{1,p}(\overline{X_\eps})}(B_\eps(x_0,r'), \overline{X_\eps}\setminus B_\eps(x_0,R'))$
and so
\begin{equation*}
\int_{\overline{X_\eps}} g_{Eu}^p\, d\mube \ge \cp_{B^\theta_{p,p}(Z)}(B_\eps(x_0,r'), \overline{X_\eps}\setminus B_\eps(x_0,R')).
\end{equation*}

Recall from~\cite[Lemma~10.6]{B2S} that as $Z$ is Ahlfors $Q$-regular, we have a lower mass bound exponent for
$\mube$ on $\overline{X_\eps}$
given by $Q_\beta:=\max\{1,Q+\tfrac{\beta}{\eps}\}$. Also from Proposition~\ref{prop:TraceExt} we know that
$(\overline{X_\eps},d_\eps,\mube)$ is doubling and supports a $1$-Poincar\'e inequality; hence we are in a position
to apply~\cite[Lemma~9.3.6]{HKST} together with~\cite{KS}, to obtain that 
\[
\cp_{N^{1,p}(\overline{X_\eps})}(B_\eps(x_0,r'), \overline{X_\eps}\setminus B_\eps(x_0,R'))\ge C(R',r'),
\]
where 
\begin{enumerate}
\item if $1 < p<Q_\beta$, 
then $Q_\beta=Q+\beta/\eps>1$ and thus
\begin{align*}
 C(R',r')&\simeq \frac{\mube(B_\eps(x_0,r'))^{1-\frac{p}{Q_\beta}}\, \mube(B_\eps(x_0,R'))^{\frac{p}{Q_\beta}}}
    {(R')^p}\\
    &\simeq (r')^{(Q+\frac\beta\eps)(1-\frac{p}{Q_\beta})}\, (R')^{\frac{p}{Q_\beta}(Q+\frac\beta\eps)-p}
\simeq r^{Q+\tfrac\beta\eps-p}.
\end{align*}
\item if $1 <  p=Q_\beta$, then again $Q_\beta=Q+\beta/\eps$ and so
\begin{align*}
C(R',r')&\simeq \frac{\mube(B_\eps(x_0,R'))}{(R')^{Q_\beta}}\, 
 \left(\log\left(\frac{C\, \mube(B_\eps(x_0,R'))}{\mube(B_\eps(x_0,r'))}\right)\right)^{1-Q_\beta}\\
 & \simeq (R')^{Q+\frac\beta\eps-Q_\beta}\, \left(\log\left(C\tfrac{R'}{r'}\right)\right)^{1-Q_\beta}
  \simeq \left(\log \tfrac{R}{r}\right)^{1-p}.
\end{align*}
In the last step we need the assumption that $R/r\ge 2$.
\item if $p>Q_\beta$, then 
\[
C(R',r')\simeq \frac{\mube(B_\eps(x_0,R'))}{(R')^p}\simeq (R')^{Q+\frac\beta\eps-p}\simeq R^{Q+\frac\beta\eps-p}.
\]
\end{enumerate}
In the above cases we also used~\eqref{eq:co-dim} and the fact that $\nu$ is Ahlfors $Q$-regular. 
Note that the balls $B_\eps(x_0,R)$ and $B_\eps(x_0,r)$ are balls centered at the point $x_0\in Z$.
(Alternatively, similar estimates as above can be obtained by applying the 
capacity estimates given in~\cite[Sections~6 and~7]{BBL}.)

From the above estimates and \eqref{eq:norm-gEu}
we conclude that 
\[
\Vert u\Vert^{p}_{B^{\theta}_{p, p}(Z)}\ge C(R,r),
\]
where $C(R,r)$ has the desired forms as in the statement of Theorem~\ref{thm:Ahlfors-Annulus}
since $\theta p = p-\beta/\eps$.
The claim follows by  
taking the infimum over all such capacity test functions $u$.
\end{proof}

\subsection{Loewner-type bounds for Besov capacity}\label{SubSec:Loewner}

Next we obtain an estimate for the Besov capacity associated to two compact continua $E,F$,
given in terms of their Hausdorff contets.
Recall the definition
of Hausdorff content from Subsection~\ref{SubSec:Cap}.

\begin{theorem}\label{thm:Loewner}
Assume that $Z$ is a compact metric space
and that $\nu$ is an Ahlfors $Q$-regular measure on $Z$, for some $Q>0$.
Let $x_0\in Z$, $R>0$ and $0< s < Q$. 
Suppose also that $E,F$ are two disjoint compact 
sets such that $E,F\subset B(x_0,R)$. 
Then for each $p>\max\{1, Q-s\}$ and for each $\theta$ satisfying
$\tfrac{Q-s}{p}<\theta<1$, we have 
\[
\cp_{B^\theta_{p,p}(Z)}(E,F) \gtrsim \frac{\mathcal{H}^s_\infty(E) \wedge \mathcal{H}^s_\infty(F)}{R^{s-Q+\theta p}}.
\]
\end{theorem}

The proof of the theorem, given next, is modeled after the corresponding result for 
Sobolev capacities found in~\cite{HK}.

\begin{proof}
Fix $p>1$ such that $\frac{Q-s}{p}<1$, and let $\theta>0$ be such that $\tfrac{Q-s}{p}<\theta<1$. 
Choose $\beta>0$ in the hyperbolic filling construction given in Subsection~\ref{SubSec:HypFill}
so that $\theta =1-\beta/(p \eps)$. Then, because of the condition that
$Q-s<\theta p$, necessarily $p+s-Q-\beta/\eps>0$. 
Let $u\in B^\theta_{p,p}(Z)$ such that $u=1$ 
in a neighborhood of $E$ and $u=0$ in a neighborhood of $F$, and 
let $Eu$ be the extension of $u$ to the 
uniformization $X_\eps$ of the
hyperbolic filling $X$ of $Z$ as 
explained in Section~\ref{subSec:setting}. 
Then
$Eu \in \Np(\overline{X}_{\epsilon}, \mu_{\beta})$, and by~\eqref{eq:norm-gEu} we have
\[
\int_{\overline{{X}}_{\epsilon}}g_{Eu}^p \dmu_{\beta} \lesssim \Vert u\Vert^{p}_{B^{\theta}_{p, p}(Z)}.
\] 

We now proceed essentially as in~\cite[Proof of Theorem~5.9]{HK}. 
We cannot apply the theorem from~\cite{HK} directly
because we do not have knowledge of the requisite lower mass bound property for $\mu_\beta$ on $\overline{X}_\eps$.
Nevertheless, their proof does apply here because we only need to apply the lower mass bound property on balls centered at 
points in $\partial X_\eps=Z$, and for such balls we have the needed lower mass bound estimate 
from~\eqref{eq:co-dim}. 
For the convenience of the reader, we provide the complete proof here. See also~\cite{KKSS} for a similar adaptation
of~\cite{HK}.

We first show that 
\[
\frac{\mathcal{H}^s_\infty(E) \wedge \mathcal{H}^s_\infty(F)}{R^{s+p-(Q+\beta/\eps)}} 
\lesssim \int_{B(x_{0}, 4R)}g_{Eu}^p \dmu_\beta.
\]
If there exist points 
$x \in E$ and $y \in F$ such that neither $|Eu(x)-(Eu)_{B_\eps(x, R)}|$ nor $|Eu(y)-(Eu)_{B_\eps(y, 3R)}|$ exceeds $1/3$, then
\[
1 \leq |Eu(x)-Eu(y)| \leq \frac{1}{3}+|Eu_{B_\eps(x, R)}-Eu_{B_\eps(y, 3R)}|+\frac{1}{3},
\]
and so from the $1$-Poincar\'e inequality on $\overline{X}_\eps$ together with H\"older's inequality, 
the above inequality implies that 
\[
\frac13 \leq C\vint_{B_\eps(y, 3R)}|Eu-Eu_{B_\eps(y, 3R)}|\dmu_\beta 
\leq CR\biggl(\vint_{B_\eps(y, 3R)}g_{Eu}^p \dmu_{\beta}\biggr)^{1/p}.
\]
Hence from~\eqref{eq:co-dim} we get 
\begin{equation}\label{eq:Loewner-1}
\frac{\nu(B_\eps(y,R))}{R^{p-\beta/\eps}} \lesssim \frac{\mu_{\beta}(B_\eps(y, R))}{CR^p} 
\leq \int_{B_\eps(y,3R)}g_{Eu}^p \dmu_{\beta}.
\end{equation}
Then, from the Ahlfors $Q$-regularity of $\nu$, together with 
the estimates
$\mathcal{H}^s_\infty(E)\lesssim R^s$
and $\mathcal{H}^s_\infty(F)\lesssim R^s$ 
and the identity $\theta p = p -\beta/\eps$, it follows that
\[
\frac{\mathcal{H}^s_\infty(E) \wedge \mathcal{H}^s_\infty(F)}{R^{s-Q+\theta p}} 
\lesssim\frac{R^s}{R^{s-Q+p-\beta/\eps}}\lesssim \int_{B_\eps(x_{0}, 4R)}g_{Eu}^p \dmu_\beta
\]
as desired.

Now suppose that the above assumption
fails. Then either for each $x\in E$ we have 
$1/3 \leq |Eu(x)-Eu_{B_\eps(x, R)}|$, 
or else for each $y\in F$ we have $1/3\le |Eu(y)-Eu_{B_\eps(y, 3R)}|$. Suppose now that for each
$x\in E$ we have 
\[
\frac13\le |Eu(x)-Eu_{B_\eps(x,R)}|.
\]
Set $\tau := \frac{s+p-(Q+\beta/\eps)}{p}$; note that $\tau>0$. Then since $x$ is a Lebesgue point of $Eu$, we have
\begin{align*}
	C(\tau)\sum_{j=0}^{\infty}2^{-i\tau} & \lesssim \sum_{j=0}^{\infty}|Eu_{B_{j}(x)}-Eu_{B_{j+1}(x)}|
	  \lesssim  \sum_{j=0}^{\infty}2^{-j}R\biggl(\vint_{B_{j}(x)} g_{Eu}^p \dmu_\beta\biggr)^{1/p}\\
	&\lesssim \sum_{j=0}^{\infty}{(2^{-j}R)^{1-(Q+\beta/\eps)/p}}\biggl(\int_{B_{j}(x)} g_{Eu}^p \dmu_\beta\biggr)^{1/p},
\end{align*}
where $B_{j}(x):=B(x, 2^{-j}R)$. Here we also
used the fact that for balls $B_\eps(x,\rho)$ with $x\in Z$ and
$0<\rho\le \diam(Z)$ we have
$\mu_\beta(B_\eps(x,\rho)\simeq \rho^{Q+\beta/\eps}$.
Hence there exists $j_{x} \in \mathbb{N}\cup \{0\}$ such that
\begin{equation}\label{eq:Loewner-2-beginning}
	2^{-j_{x}\tau p} \lesssim (2^{-j_x}R)^{p-(Q+\beta/\eps)} \int_{B_{j_{x}}(x)} g_{Eu}^p \dmu_\beta.
\end{equation}
The above inequality, together with 
our choice of
$\tau$, gives
\[
	2^{-j_{x}s} \lesssim R^{p-(Q+\beta/\eps)}\int_{B_{j_{x}}(x)} g_{Eu}^p \dmu_\beta.
\]
By the $5$-covering Lemma~\cite{Heinonen}
there exists a countable pairwise disjoint family of balls  
$\{B(x_k,  2^{-j_{x_k}}R)\}_{k \in \mathbb{N}}$ such that
\[
E \subseteq \bigcup_{k} B(x_k, 2^{-j_{x_k}}5R)
\]
and 
\begin{equation}\label{local estimate 1}
	2^{-j_{x_k}s} \lesssim R^{p-(Q+\beta/\eps)}\int_{B_{j_{x_k}}(x_k)} g_{Eu}^p \dmu_\beta.
\end{equation}
Hence, by \eqref{local estimate 1} and the pairwise disjointness property, we have
 \[
 \mathcal{H}^s_\infty(E)\leq C \sum_{k=1}^{\infty} (2^{-j_{x_k}}R))^s 
   \lesssim R^{s+p-(Q+\beta/\eps)} \int_{B(x_{0},4R)}g_{Eu}^p \dmu_\beta.
 \]

A similar argument shows that if 
for each $y\in F$ we have
 \[
 \frac13\le |Eu(y)-Eu_{B(y, 3R)}|,
 \]
then
\[
\mathcal{H}^s_\infty(F)  
\lesssim R^{s+p-(Q+\beta/\eps)} \int_{B(x_{0}, 4R)}g_{Eu}^p \dmu_\beta.
\]
Combining the two possibilities 
and applying the identity $\theta p = p -\beta/\eps$, we see that
\begin{equation}\label{eq:Loewner-2}
\frac{\mathcal{H}^s_\infty(E) \wedge \mathcal{H}^s_\infty(F)}{R^{s-Q+\theta p}} 
\lesssim \int_{B(x_{0},4R)}g_{Eu}^p \dmu_\beta
\end{equation}
as desired.

The proof is completed by 
first recalling from~\eqref{eq:norm-gEu} that 
$\int_{X_\eps} g_{Eu}^p\, d\mu_\beta \lesssim \Vert u\Vert_{B^\theta_{p,p}(Z)}$, and then
taking the infimum over all capacity test functions $u$
in the above two cases.
\end{proof}

If $E$ and $F$ are connected sets and $s=1$, then $\mathcal{H}^s_\infty(E)\simeq\diam(E)$ and
$\mathcal{H}^s_\infty(F)\simeq\diam(E)$. If they are not necessarily connected but
$\nu(E)>0$ and $\nu(F)>0$, then for each $0<s<Q$ we have that
$\mathcal{H}^s_\infty(E)\ge \nu(E)\, R^{s-Q}$ and $\mathcal{H}^s_\infty(F)\ge\nu(F)\, R^{s-Q}$.

\section{$B^{\theta}_{p,p}$-morphisms and quasisymmetric maps}\label{Sec:QC}

From~\cite[Theorem~1.1]{Cp} it is known that there is a correspondence between quasisymmetric mappings between
two Ahlfors regular compact metric spaces and certain classes of weights 
on the hyperbolic fillings of either of the metric spaces.
The perspective of~\cite{KKSS, KYZ} is different in that unlike~\cite{Cp}, they consider impact of quasisymmetric mappings
on the relevant Besov classes of functions on the metric spaces themselves.

In this section, we extend the theory from~\cite{KKSS} to Ahlfors regular spaces which do not
support any Poincar\'e inequalities, see Theorem~\ref{thm:QS-Besov} below. 
We begin by recalling the definitions of quasisymmetry.

\begin{definition}
Let $(Z,d_Z)$ and $(W,d_W)$ be metric spaces.
\begin{itemize}
\item[(a)] A homeomorphism $\pip:Z\to W$ is a \emph{quasisymmetric map}
if there is a continuous monotone increasing function $\eta:[0,\infty)\to[0,\infty)$ with $\eta(0)=0$ and $\eta(t)>0$ when
$t>0$, such that for each triple of points $x,y,z\in Z$ we have
\[
\frac{d_W(\pip(x),\pip(z))}{d_W(\pip(x),\pip(y))}\le \eta\left(\frac{d_Z(x,z)}{d_Z(x,y)}\right).
\]

\item[(b)] A homeomorphism $\pip:Z\to W$ is \emph{weakly quasisymmetric} if there is some $H>0$ such that
for each triple of points $x,y,z\in Z$ we have
\[
\frac{d_W(\pip(x),\pip(z))}{d_W(\pip(x),\pip(y))}\le H\ \text{ whenever }\ 
\frac{d_Z(x,z)}{d_Z(x,y)}\le 1.
\]
In addition, $\pip$ is \emph{uniformly locally weakly quasisymmetric} if there is some $\rho>0$ such that the restriction of $\pip$ to balls in $Z$ of radii at most $\rho$ are weakly quasisymmetric with the same constant $H$.
\end{itemize}
\end{definition}

\begin{remark}\label{rmk:loc-qs}
From~\cite[Theorem~10.19]{Heinonen} we know that if both $Z$ and $W$ are connected doubling metric spaces,
then weak quasisymmetry is equivalent to quasisymmetry. 
Moreover, the proof given there works even if $\pip$ is only
known to be uniformly locally weakly quasisymmetric;
this is seen as follows.

From uniformly locally weak quasisymmetry, together with the connectendess
property, we know that
the homeomorphism is promoted to uniformly local quasisymmetry; that is, there is some $r_0>0$ such that whenever
$x,y,z\in Z$ are three distinct points such that $\diam\{x,y,z\}\le r_0$, we have
\[
\frac{d_W(\pip(x),\pip(y))}{d_W(\pip(x),\pip(z))}\le \eta\left(\frac{d_Z(x,y)}{d_Z(x,z)}\right).
\]
Since $Y$ is compact and $\pip^{-1}$ is continuous, it follows that there is some $\kappa>0$ such that for all $x,y\in X$ we have that
$d_W(\pip(x),\pip(y))\ge \kappa$ whenever $d_Z(x,y)\ge r_0/4$. If $x,y,z\in X$ are three distinct points such that 
$d_Z(x,z)\le r_0/2$ and $d_Z(x,y)>r_0/2$, then by the connectedness property of $Z$ we can find $w_0\in Z$ such that
$d_Z(x,y_0)=r_0/2$, and so by the monotonicity of the quasisymmetry gauge $\eta$,
\begin{align*}
\frac{d_W(\pip(x),\pip(y))}{d_W(\pip(x),\pip(z))}
  &=\frac{d_W(\pip(x),\pip(y))}{d_W(\pip(x),\pip(y_0))}\ \frac{d_W(\pip(x),\pip(z))}{d_W(\pip(x),\pip(y_0))}\\
  &\le \frac{\diam(W)}{\kappa}\ \eta\left(\frac{d_Z(x,y_0)}{d_Z(x,z)}\right)
  \le \frac{\diam(W)}{\kappa}\ \eta\left(\frac{d_Z(x,y)}{d_Z(x,z)}\right).
\end{align*}
Moreover,
\begin{align*}
\frac{d_W(\pip(x),\pip(z))}{d_W(\pip(x),\pip(y))}
  &=\frac{d_W(\pip(x),\pip(z))}{d_W(\pip(x),\pip(y_0))}\ \frac{d_W(\pip(x),\pip(y_0))}{d_W(\pip(x),\pip(y))}\\
  &\le \frac{\diam(W)}{\kappa}\ \eta\left(\frac{d_Z(x,z)}{d_Z(x,y_0)}\right)\\
  &\le \frac{\diam(W)}{\kappa}\ \eta\left(\frac{d_Z(x,z)}{d_Z(x,y)}\frac{d_Z(x,y)}{d_Z(x,y_0)}\right)\\
  &\le \frac{\diam(W)}{\kappa}\ \eta\left(\frac{2\diam(X)}{r_0}\frac{d_Z(x,z)}{d_Z(x,y)}\right).
\end{align*}
Finally, if $d_Z(x,y)\ge r_0/2$ and $d_Z(x,z)\ge r_0/2$, then by the monotonicity of $\eta$ again,
\[
\frac{d_W(\pip(x),\pip(z))}{d_W(\pip(x),\pip(y))}
 =\frac{\diam(W)}{\kappa}
 \le \frac{\diam(W)}{\kappa}\, \frac{\eta\left(\frac{d_Z(x,z)}{d_Z(x,y)}\right)}{\eta\left(\frac{r_0}{2\diam(Z)}\right)}.
\]
It follows that $\pip$ is globally quasisymmetric as well, with quasisymmetry gauge $\widehat{\eta}$ given by
\[
\widehat{\eta}(t)=\max\bigg\lbrace\eta(t),\tfrac{\diam(W)}{\kappa}\eta(t), \tfrac{\diam(W)}{\kappa}\eta(\tfrac{2\diam(Z)}{r_0}t),
   \tfrac{\diam(W)}{\kappa\, \eta\left(\tfrac{r_0}{2\diam(Z)}\right)}\, \eta(t)\bigg\rbrace.
\]
\end{remark}

\begin{theorem}\label{thm:QS-Besov}
Assume that $(Z,d_Z,\nu_Z)$ and $(W,d_W,\nu_W)$
are compact metric measure spaces, with 
$\nu_Z$ Ahlfors $Q_Z$-regular and $\nu_W$ Ahlfors $Q_W$-regular
for some $Q_Z,Q_W>0$.
Suppose that a homeomorphism $\pip:Z\to W$ induces a bounded linear operator
$\pip_\#:B^{\theta_W}_{p,p}(W)\to B^{\theta_Z}_{p,p}(Z)$, that is, 
there is a constant $C_\pip>0$ such that whenever
$f\in B^{\theta_W}_{p,p}(W)$ we have that $f\circ\pip\in B^{\theta_Z}_{p,p}(Z)$ with
\[
\Vert f\circ\pip\Vert_{B^{\theta_Z}_{p,p}(Z)}
  \le C\, \Vert f\Vert_{B^{\theta_W}_{p,p}(W)},
\]
where $\theta_Z=Q_Z/p$ and $\theta_W\le Q_W/p$.
Suppose in addition that $W$ is linearly locally path-connected, that is, 
there is a constant $C_L>1$ such that given $w\in W$, $0<r<\diam(W)$, and 
$w_1,w_2\in B(w,r)\setminus B(w,r/2)$ there is a path $\gamma$ in $B(w,C_Lr)\setminus B(w,r/C_L)$ with 
end points $w_1, w_2$. Then $\pip$ is a quasisymmetric map.
\end{theorem}

Here we should be careful in stating what $f\circ\pip$ is, as it may be the case that $\pip$ pulls back a set of $\nu_W$-measure zero
to a set of positive $\nu_Z$-measure. Instead, we here require that we only consider the Besov quasicontinuous $f$ in
looking at $f\circ\pip$. Such quasicontinuous representatives of functions in $B^{\theta_W}_{p,p}(W)$ (which are,
strictly speaking, equivalence classes of functions) are guaranteed to exit, thanks to the results in~\cite{B2S}.

The argument below is very similar to that of~\cite{HK} where both the metric measure spaces
are assumed to be connected and uniformly locally Ahlfors $Q$-regular, 
and to support a uniformly local $Q$-Poincar\'e inequality.

\begin{proof}[Proof of Theorem~\ref{thm:QS-Besov}]
Since $W$ is connected, therefore $Z$ is also connected, and so 
by Remark~\ref{rmk:loc-qs},
it suffices to show that $\pip$ is uniformly locally weakly
quasisymmetric. 

Let $\pip$ be as in the statement of the theorem. Since
$\pip$ is continuous on the compact space $Z$, it is uniformly continuous. Hence
we can find $R_0>0$ such that whenever $x_1,x_2\in Z$ with $d(x_1,x_2)\le R_0$ we have that 
$d(\pip(x_1),\pip(x_2))<\diam(W)/10C_L^4$. By choosing $R_0$ small, we can also ensure that
$R_0\le \diam(Z)/10$.

We fix $x\in Z$ and consider $y,z\in Z$ such that $r:=d(x,y)\le d(x,z)=:R< R_0$. 
We wish to find an upper bound for
$d(\pip(x),\pip(y))/d(\pip(x),\pip(z))$. Set $L=d(\pip(x),\pip(y))$ and $l=d(\pip(x),\pip(z))$. If $L\le 4C_L^2l$, then we have
a bound in terms of $4C_L^2$. So suppose that $L>4C_L^2l$. By the choice of $R_0$ we can find $w\in W$ such that
$d(\pip(x),w)>C_L^2L$; then $d(\pip^{-1}(w),x)>R_0$. 

Let $E, F\subset W$ such that $E$ is a curve in $W\setminus B(\pip(x), 2C_Ll)$ with end points $w,\pip(y)$ and
$F$ is a curve in $B(\pip(x),C_Ll)$ with end points $\pip(x), \pip(z)$; these curves are guaranteed by the linear
local path-connectedness of $W$. Then by Theorem~\ref{thm:Ahlfors-Annulus},
\begin{align*}
\cp_{B^{\theta_W}_{p,p}(W)}(E,F)\le 
\cp_{B^{\theta_W}_{p,p}(W)}(B(\pip(x),C_Ll),W\setminus B(\pip(x), L/C_L))
\le C\, \log(L/(C_L^2l))^{\beta_p}
\end{align*}
where $\beta_p\in \{-p,1-p\}$.
By the assumed morphism property of $\pip$, it follows that
\[
\cp_{B^{\theta_Z}_{p,p}(Z)}(E',F')\le C\, \log(L/(C_L^2l))^{\beta_p},
\]
where $E'=\pip^{-1}(E)$ and $F'=\pip^{-1}(F)$. On the other hand, both $E'$ and $F'$ are connected subsets of
$\pip^{-1}(B(\pip(x),C_Ll))$ and $\pip^{-1}(W\setminus B(\pip(x), L/C_L))$. Moreover, $F'$ contains both $x$ and $z$,
while $E'$ contains both $y$ and $w$. It follows that 
\[
\min\{\mathcal{H}^1_\infty(E'\cap B(x,2r)),\mathcal{H}^1_\infty(F'\cap B(x,2r))\}\ge r, 
\]
and so
by Theorem~\ref{thm:Loewner} and the assumption $\theta_Zp=Q_Z$, we have 
\[
\cp_{B^{\theta_Z}_{p,p}(Z)}(E',F')\ge 1/C.
\]
It follows that 
$\log(L/l)^{p-1}\le C$, that is,
\[
L\le e^{C^{-1/\beta_p}}\, l,
\]
as desired.
\end{proof}

From~\cite{KKSS} we know that if $Q_Z=Q_W$ and $Z$ supports a $Q$-Poincar\'e inequality, then every quasiconformal map
$\pip: Z\to W$ is also a $B^{Q/p}_{p,p}$-morphism, that is,
$\pip$ induces a bounded linear operator $\pip_\#:B^{\theta_W}_{p,p}(W)\to B^{\theta_Z}_{p,p}(Z)$.
In our setting we do not know whether this converse of Theorem~\ref{thm:QS-Besov}
holds even if the quasiconfomal map is a quasisymmetric map. The principal
stumbling block in this case is our lack of knowledge of absolute continuity of the pull-back measure with respect to the 
underlying measure, namely, whether $\pip_\#\nu_W$ is absolutely continuous with respect to $\nu_Z$, with appropriate
integrability conditions of the Jacobian. It is however possible that even 
with such lack of absolute continuity 
we do obtain a morphism if, perhaps, we focus on quasicontinuous representative Besov functions.

\vskip .3cm

\noindent {\bf Address:}\\

\noindent {\bf J.L.}: University of Jyvaskyla, Department of Mathematics and Statistics, P.~O.~Box 35, FI-40014 University of Jyvaskyla, Finland.\\

\noindent {\bf E-mail address:} {\tt juha.lehrback@jyu.fi}\\

\noindent {\bf N.S.}: Department of Mathematical Sciences, P.~O.~Box 210025, University of Cincinnati, Cincinnati, 
OH~45221-0025, U.S.A.\\

\noindent {\bf E-mail address:} {\tt shanmun@uc.edu}

\end{document}